\documentclass[reqno,a4paper, roman, 11pt]{amsart}

\usepackage{bm}
\usepackage{amsmath}
\usepackage{amsfonts}
\usepackage{amssymb}
\usepackage{amsthm}
\usepackage{graphicx}
\usepackage{color}
\usepackage{enumitem}
\usepackage{url}
\usepackage{a4wide}
\theoremstyle{plain}

\newtheorem{theorem}{Theorem}[section]
\newtheorem{corollary}[theorem]{Corollary}
\newtheorem{definition}[theorem]{Definition}

\newtheorem{lemma}[theorem]{Lemma}
\newtheorem{notation}[theorem]{Notation}
\newtheorem{proposition}[theorem]{Proposition}
\newtheorem{remark}[theorem]{Remark}

\numberwithin{equation}{section}
\newcommand{\tr}{\mathop{\mathrm{tr}}}

\newcommand{\diag}{\mathop{\mathrm{diag}}}
\newcommand{\rank}{\mathop{\mathrm{rank}}}

\usepackage{hyperref}

\begin{document}
\title[Non-central Wishart distributions]{On the existence of non-central Wishart distributions}
\author{Eberhard Mayerhofer}
\address{Deutsche Bundesbank, Research Centre, Wilhelm-Epstein-Str.14, 60431 Frankfurt am Main, Germany}
\email{eberhard.mayerhofer@gmail.com}
\thanks{The author currently holds a Marie-Curie fellowship at Deutsche Bundesbank. The research leading to these results
has received funding from the European Community Programme FP7-PEOPLE-ITN-2008 under
grant agreement number PITN-GA-2009-237984 (RISK) and from the WWTF (Vienna Science and Technology Fund) 
 The funding is gratefully acknowledged.}
\begin{abstract}
This paper deals with the existence issue of non-central Wishart
distributions which is a research topic initiated by Wishart (1928),
and with important contributions by e.g., L\'evy (1937), Gindikin
(1975), Shanbhag (1988), Peddada \& Richards (1991). We present a
new method involving the theory of affine Markov processes, which
reveals joint necessary conditions on shape and non-centrality
parameter. While Eaton's conjecture concerning the necessary range
of the shape parameter is confirmed, we also observe that it is not
sufficient anymore that it only belongs to the Gindikin ensemble, as
is in the central case. 
\end{abstract}
\maketitle

\section{Introduction}
The general non-central Wishart distribution
$\Gamma(p,\omega;\sigma)$ on the cone $S_d^+$ of symmetric positive
semi-definite $d\times d$ matrices is
 defined (whenever it exists) by its Laplace transform
\begin{equation}\label{FLT Mayerhofer Wishart}
\mathcal L (\Gamma(p,\omega;\sigma))(u)= \left(\det(I+\sigma u)\right)^{-p}e^{-\tr(u(I+\sigma
u)^{-1}\omega)},\quad u\in S_d^+,
\end{equation}
were $p\geq 0$ denotes its shape parameter, $\sigma\in S_d^{+}$ is the
scale parameter and the parameter of non-centrality equals
$\omega\in S_d^+$. In the case that $\omega=0$,
$\Gamma(p;\sigma):=\Gamma(p,0;\sigma)$ is called the central Wishart
distribution, which had been introduced in 1928 by Wishart
\cite{Wishart}. In 1937, L\'evy \cite{Levy} showed that
$\Gamma(p;\sigma)$ on $S_2^+$ is not infinitely divisible for
invertible $\sigma$, which means that for some sequence of shape
parameters $p_k\downarrow 0$, $\Gamma(p_k;\sigma)$ cannot exist.
Gindikin \cite{Gindikin}, Shanbhag \cite{Shanbhag} and Peddada \&
Richards \cite{PeddadaRichards91}\footnote{Contrary to
\cite{PeddadaRichards91} we exclude the point mass at zero, i.e. the
Gindikin ensemble does not contain $0$. Also, our notation deviates slightly from theirs, see Section \ref{App}.} subsequently showed that
for non-degenerate $\sigma$,
\begin{equation}\label{eq: Laplacecentralwishart}
\left(\det(I+\sigma u)\right)^{-p}
\end{equation}
can only be the Laplace transform of a non-trivial probability
measure for shape parameters $p$ belonging to the Gindikin ensemble
\[
\Lambda_d=\left\{\frac{j}{2},\quad j=1,2,\dots,
d-2\right\}\cup\left[\frac{d-1}{2},\infty\right).
\]
Aim of this work is to investigate this fundamental existence issue
in the non-central case. We shall show:
\begin{theorem}\label{th: maintheorem}
Let $d\in\mathbb N$, $p> 0$, $\omega,\sigma \in S_d^+$. The
following hold:
\begin{enumerate}
\item \label{part1thmain} Suppose $\sigma$ is invertible. If the right side of \eqref{FLT Mayerhofer Wishart} is the Laplace transform of a non-trivial probability
measure \footnote{It is easy to see that if $\sigma\neq 0$, the
triviality of $\Gamma(p,\omega;\sigma)$ is equivalent to $p=0$, and
$\omega=0$ (in which case we have the point mass at $0$).}
$\Gamma(p,\omega;\sigma)$ on $S_d^+$, then $p\in \Lambda_d$ and
$\rank(\omega)\leq 2p+1$.
\item Conversely, suppose any of the following conditions hold:
\begin{enumerate} 
\item $p\geq \frac{d-1}{2}$,
\item $p<\frac{d-1}{2}$ and $\rank(\omega)\leq 2p$. 
\end{enumerate}
Then the right side
of \eqref{FLT Mayerhofer Wishart} is the Laplace transform of a
non-trivial probability measure $\Gamma(p,\omega;\sigma)$.
\end{enumerate}
\end{theorem}
It should be noted that Theorem is not a full characterization of the existence of non-central Wishart distributions,
because it leaves open the question, whether distributions  $\Gamma(p,\omega;\sigma)$ exist with $p\in\{1/2,\dots,\frac{d-2}{2}\}$
and $\rank(\omega)=2p+1$. This is only an interesting question for $d\geq 3$, and $\rank(\omega)>1$
as the following two corollaries demonstrate. These are immediate conclusions of Theorem \ref{th: maintheorem}.
\begin{corollary}\label{corollary eaton}
Suppose that $\sigma$ is invertible and
$\rank(\omega)\leq 1$. The following are equivalent:
\begin{enumerate}
\item The right side of \eqref{FLT Mayerhofer Wishart} is the Laplace transform of a non-trivial probability measure
$\Gamma(p,\omega;\sigma)$ on $S_d^+$.
\item $p\in \Lambda_d$.
\end{enumerate}
\end{corollary}
Another trivial consequence holds in low dimensions. Note that
$\Lambda_1=[0,\infty)$ and $\Lambda_2=[\frac{1}{2},\infty)$:
\begin{corollary}\label{corollarydim2}
Let $d\leq 2$, and suppose $\sigma$ to be invertible. The following
are equivalent:
\begin{enumerate}
\item The right side of \eqref{FLT Mayerhofer Wishart} is the Laplace transform of a non-trivial probability measure
$\Gamma(p,\omega;\sigma)$ on $S_d^+$.
\item $p\in \Lambda_d$.
\end{enumerate}
\end{corollary}

We slightly adapt the notation of the recent article by
Letac and Massam \cite{Letac08}, and we are recollecting a
number of fundamental statements thereof below (see section \ref{subsec
1}) especially what concerns basic properties of non-central Wishart
distributions. Concerning their main statement as well as that of
\cite{PeddadaRichards91}, the following important remark is due:
\begin{remark}\rm
\begin{enumerate}
\item \cite[Proposition 2.3]{Letac08} claims that  $\Lambda_d$ fully characterizes the existence of non-central Wishart distributions.
But this (paradoxically) allows the construction of Markovian Feller
semigroups on $S_d^+$ which are non-positive, a mere impossibility.
That's how we obtain the additional necessary conditions on the rank
of $\omega$ in dependence of $p$, which suggests that the
characterization of \cite{Letac08} is wrong. On the other hand,
it is obvious that the existence proof of \cite[Proposition 2.3, see also Proposition 2.1 and the subsequent paragraph]{Letac08} is
incomplete, as for $p<\frac{d-1}{2}$ and $\rank(\omega)>2p$, the
existence of non-central Wishart distributions
    $\Gamma(p,\omega;\sigma)$ is not shown there.
\item \cite[Theorem 1]{PeddadaRichards91} prove the necessity of
$p\in\Lambda_d$ (which had been a conjecture by M.L. Eaton) under the premise that $\rank(\omega)=1$. However,
the method of \cite{PeddadaRichards91}, which involves the theory of
zonal polynomials, relies on the non-negativity of the so-called
generalized binomial coefficients which may be ''difficult to prove``
in the case that $\rank(\omega)>1$, see their concluding remark in
    \cite[Section 4]{PeddadaRichards91}. In contrast, the present
    paper shows with a much simpler argument that $p\in\Lambda_d$,
    for {\it all} non-central Wishart distributions with
    nondegenerate scale parameter (see the first part of the proof
    of Theorem \ref{th: maintheorem} \ref{part1thmain}.
    \end{enumerate}
\end{remark}
Our method for approaching the existence issue is new. We shall see
that $p\in\Lambda_d$ can be proved by utilizing the situation in the
central case (restated as Theorem \ref{th: mainproposition}) and
L\'evy's continuity theorem as well as a number of elementary facts
such as (1) the behaviour of $\Gamma(p,\omega;\sigma)$ under the the
action of the linear automorphism group of $S_d^+$ and (2) the
characterizing property of the natural exponential family associated
with $\Gamma(p,\omega;\sigma)$, see Proposition \ref{proplemma}.
Concerning the rank condition, we proceed with an indirect argument:
Assuming by contradiction that $p<\frac{d-1}{2}$ and $\rank(\omega)>
2p+1$ allows  the existence of affine Feller diffusions $X$
supported on $S_d^+$ (see \cite{CFMT}) whose transition laws are
non-centrally Wishart distributed. The contradiction
$p\geq\frac{d-1}{2}$ is derived by observing that $X$ violates a
(geometric relevant) drift condition, as established in \cite{CFMT}.
The latter is a consequence of the fact that the infinitesimal
generator of a Markovian Feller semigroup satisfies the strong
maximum principle, see \cite[section 4.4]{CFMT}.
\subsection{Program of the paper}
In section  \ref{sec: prelim} we deliver notation and recall known
facts about the existence of non-central Wishart laws (subsection
\ref{subsec 1}) and Wishart processes on $S_d^+$ (subsection
\ref{subsec 2}). The latter section uses a convenient notation for
the Laplace transform, such that the distribution of Wishart
processes can be easily read off from the characteristic exponents
of the (affine) process, and which could be easily turned into an
existence proof alternative to the one of \cite{CFMT}. The
presentation of section \ref{subsec 2} is instructive, and is of
relevance for the proof of Theorem \ref{th: maintheorem}. Also, for
the sake of completeness, we restate the characterization of the
central Wishart distributions in terms of the Gindikin ensemble in
subsection \ref{subsec 3} and state trivial conclusions when
$\sigma$ is degenerate (characterization of existence and infinite
divisibility). Section \ref{sec: proof} presents a proof of Theorem
\ref{th: maintheorem}. In Appendix \ref{App} the relation of our definition
of non-central Wishart distributions to others in the literature is given.
\section{Notation and preliminary results}\label{sec: prelim}
\begin{notation}\rm
Throughout the present article the following notation is relevant:
\begin{itemize}
\item $\mathbb R_+$ is the non-negative real line, and $\mathbb R_{++}$ is its interior,
 \item $M_d$ denotes the set of real $d\times d$ matrices, and $S_d$ all symmetric ones therein.
 \item $I$ is the unit $d\times d$ matrix.
 \item $S_d^+$ is the cone of symmetric positive semi-definite matrices, and $S_d^{++}$
 denotes its interior, the symmetric positive definite matrices. We denote its boundary $S_d^+\setminus S_d^{++}$ by $\partial S_d^+$.
\item $\tr(A)$ is the trace of a matrix $A\in M_d$, which introduces a scalar product on
 $S_d$ via  $\langle x,y\rangle:=\tr(xy)$ for $x,y\in S_d$.
\item For $k=1,2,\dots d-1$, $D_k\subseteq \partial S_d^{+}$ determines the (non-convex) cones of $d\times d$ matrices
of rank less or equals $k$. We note that the dual cone of $D_k$ equals $D_k^\star=S_d^+$, hence the Laplace transform of
a finite measure $\mu(d\xi)$ supported on $D_k$ is defined by
\[
\mathcal L(\mu)(u):=\int_{D_k}e^{-\langle u,\xi\rangle}\mu(d\xi),\quad u\in S_d^+.
\]
\end{itemize}
\end{notation}
\subsection{Facts on non-central Wishart laws}\label{subsec 1}
First we recall the existence and basic properties of non-central Wishart distributions:
\begin{lemma}\label{lemma: existence wishart dist}
Let $p\in \Lambda_d$, $\sigma\in S_d^+$ and $\omega\in S_d^+$. We have:
\begin{enumerate}
\item \label{ex wis dist minilemma} Suppose $w=mm^\top$ for $m\in\mathbb R^d$ and set $\Sigma:=\sigma/2$. If $Y\sim \mathcal N(m, \Sigma)$,
then  $X:=YY^\top\sim \Gamma(1/2,w;\sigma )$ is supported on $D_{1}$.
\item \label{ex wis dist 0} If $p<\frac{d-1}{2}$ and $\rank(\omega)\leq 2p$, then the right side of \eqref{FLT Mayerhofer Wishart} is the Laplace transform of a probability measure supported on $D_{2p}$.
\item \label{ex wis dist 1} If $p\geq \frac{d-1}{2}$, then the right side of \eqref{FLT Mayerhofer Wishart} is the Laplace transform of a probability measure $\Gamma(p, \omega;\sigma)$ on $S_d^+$.
\item \label{ex wis dist 2} In particular, if $p>\frac{d-1}{2}$ and if $\sigma$ is invertible, then the density of $\Gamma(p, \omega;\sigma)$ exists\footnote{For a detailed exposition of the densities, which involves the zonal polynomials, we refer to \cite[eq.~p.~1400]{Letac08}} and we denote it by $F(p, \omega,\sigma,\xi)$.
\end{enumerate}
\end{lemma}
\begin{proof}
Proof of \ref{ex wis dist minilemma} is provided in \cite[Proposition 3.2]{Letac08}.
\ref{ex wis dist 0}: This follows from \ref{ex wis dist minilemma}, by taking
$2p$ independent normal random variables on $\mathbb R^d$ with distribution $\mathcal N(m_i,\Sigma)$ where
$\Sigma=\sigma/2$ and $\omega:=m_1m_1^\top+\cdots+m_{2p} m_{2p}^\top$.

Note that if $\sigma\in S_d^{++}$, our definition of non-central Wishart distribution is related to the one of
\cite{Letac08} in that $\Gamma(p,\omega;\sigma)=\gamma(p,\sigma^{-1}\omega\sigma^{-1};\sigma)$, the latter being
called ''general non-central Wishart distribution`` in \cite{Letac08} (see Appendix \ref{App} for more detailed information). Hence statement \ref{ex wis dist 2} is  a
consequence of \cite[p.\ 1400]{Letac08}.

Now for each $\varepsilon>0$ we regularize $\sigma$ and $a$ by setting
\[
\sigma_\varepsilon:=\sigma+\varepsilon I, \quad a_\varepsilon:=(\sigma+\varepsilon I)^{-1}\omega (\sigma+\varepsilon I)^{-1}.
\]
Then for each $\varepsilon>0$, we pick $X_\varepsilon$, an $S_d^+$ valued random variable according to
\cite[Proposition 2.3]{Letac08} such that
\[
X_\varepsilon\sim\Gamma(p,\omega;\sigma_\varepsilon) (=\gamma(p,a_\varepsilon;\sigma_\varepsilon)).
 \]
Letting $\varepsilon\rightarrow 0$ and using L\'evy's continuity theorem,
we infer that $X_\varepsilon$ converges in distribution to
some random variable $X\sim\Gamma(p,\omega;\sigma)$. This settles part \ref{ex wis dist 1} and
\ref{ex wis dist 2}.

\end{proof}
\subsection{On the Fourier-Laplace transform of Wishart
processes}\label{subsec 2} A stochastically continuous Markov
process $(X,\mathbb P_x)_{x\in S_d^+ }$ on $S_d^+$ is called affine, if
its Laplace transform is exponentially affine in the state-variable
(see \cite{CFMT}). That is, for all $(t,x,u)\in\mathbb R_+\times
(S_d^+)^2$
\begin{equation}\label{eq char exp}
\mathbb E[e^{-\langle u,X_t\rangle}\mid X_0=x]=e^{-\phi(t,u)-\langle \psi(t,u),x\rangle}, \quad u\in S_d^+
\end{equation}
holds, where the so-called characteristic exponents $\phi$ and $\psi$ satisfy a system of generalized Riccati equations,
\begin{align}
\dot{\phi}(t,u)&=F(\psi(t,u)),\quad \phi(0,u)=0,\\
\dot{\psi}(t,u)&=R(\psi(t,u)),\quad \psi(0,u)=u,
\end{align}
and $F,\, R$ are of a specific L\'evy-Khintchine form, which is
particularly simple in the case of Wishart processes (which are pure
diffusions; for the original definition in terms of stochastic differential
equations and particular solutions of these SDEs, see \cite{bru}):
\begin{definition}\rm
An affine process $X=(X,\mathbb P_x)_{x\in S_d^+ }$ is a Wishart process on $S_d^+$ with parameter $(p\geq 0,\alpha\in S_d^+,\beta\in M_d)$, if its characteristic exponents $(\phi,\psi)$ satisfy the following Riccati equations:
\begin{align}\label{eq: phi}
\dot{\phi}(t,u)&=2p\,\langle
\alpha,\psi(t,u)\rangle,\quad\phi(0,u)=0\\\label{eq: psi}
\dot{\psi}(t,u)&=-2\psi(t,u)\,\alpha
\psi(t,u)+\psi(t,u)\beta+\beta^\top\psi(t,u),\quad\psi(0,u)=u.
\end{align}
Using notation and language from \cite[Definition 2.3 and the discussion in section 2.1]{CFMT}, the parameters satisfy the following
 \begin{itemize}
\item $\alpha$ equals the diffusion coefficient of $X$,
\item $b=2p\, \alpha$ equals the constant drift of $X$, and
\item $B(x)=\beta x+x\beta^\top$ equals the linear drift. Note that the drift enters \eqref{eq: psi} as its transpose $B^\top(u)=\beta^\top u+u\beta$.
\end{itemize}
\end{definition}
By \cite[Theorem 2.4]{CFMT} we have:
\begin{proposition}\label{prop: drift condition}
Let $\alpha\in S_d^+$ and $p\geq 0$.
\begin{enumerate}
\item If $p\geq\frac{d-1}{2}$, then for each $\beta\in M_d$, there exists a Wishart process on $S_d^+$ with parameters $(p,\alpha, \beta)$.
\item \label{prop: drift condition item 1} Conversely, let $X$ be a Wishart process with parameters $(p,\alpha, \beta)$. Then $p\geq\frac{d-1}{2}$.
\end{enumerate}
\end{proposition}

We further recall the established fact \cite[Theorem 2.7, equation (2.22)]{CFMT} that for each $x\in S_d^+$, $(X,\mathbb P_x)$ can be realized as (a weak) solution of the stochastic differential equation \begin{equation}\label{wishart full class}
dX_t=\sqrt {X_t} dB_t Q+Q^\top dB_t^\top \sqrt {X_t}+(2p\,Q^\top Q+\beta X_t+X_t\beta^\top )dt
\end{equation}
subject to $\quad X_0=x\in S_d^+$, for any $Q\in M_d$ which satisfies $Q^\top Q=\alpha$. Here $B$ is a $d\times d$ standard Brownian motion, and $\sqrt X$ denotes the unique matrix square root on the space of positive semi-definite matrices.

In the following we denote by $\omega^{\beta}_t$ the flow of the vector field $\beta x+x\beta^\top$, that is,
\[
\omega^{\beta}:\,\, \mathbb R\times S_d^+\rightarrow S_d^+,\quad\omega^{\beta}_t(x):=e^{\beta t}x e^{\beta^\top t}.
\]
Its integral $\sigma^{\beta}_t: S_d^+\rightarrow S_d^+$ for $t\geq 0$ is denoted by
\[
\sigma^{\beta}:\,\, \mathbb R_+\times S_d^+\rightarrow S_d^+,\quad\sigma^{\beta}_t(x)=2\int_0^t \omega^{\beta}_s(x) ds.
\]
\begin{proposition}\label{prop wish proc}
Let $(X,\mathbb P_x)_{x\in S_d^+}$ be a Wishart process with parameter $(p,\alpha,\beta)$. Then the characteristic exponents $\phi,\,\psi$ take the form
\begin{align}\label{formula: Wishartphi}
\phi(t,u)&=p\log\det \left(I+u
\sigma_t^\beta(\alpha)\right),\\\label{formula: Wishartpsi}
\psi(t,u)&=e^{\beta^\top
t}\left(u^{-1}+\sigma_t^\beta(\alpha)\right)^{-1}e^{\beta t}.
\end{align}
Consequently, the Fourier-Laplace transform of $X$ is given by
\begin{equation}\label{eq: FLT Bru}
\mathbb E_x[e^{-\langle z,
X_t\rangle}]=\left(\det(I+\sigma^{\beta}_t(\alpha)z) \right)^{-p}
e^{-\tr\left(z\left(I+\sigma^{\beta}_t(\alpha)
z\right)^{-1}\,\omega^{\beta}_t(x)\right)},
\end{equation}
for all $z\in S_d^++i S_d$.
\end{proposition}
\begin{proof}
We first solve the generalized Riccati equations
\eqref{eq: phi}--\eqref{eq: psi} for initial data $u\in S_d^+$.
Formula \eqref{formula: Wishartpsi} for $\psi$ follows from the fact that
$\frac{d}{dt}a^{-1}(t)=-a^{-1}(t) \frac{d}{dt}a(t)a^{-1}(t)$, see
\cite[Proposition III.4.2 (ii)]{Faraut}. Formula \eqref{formula: Wishartphi} follows by some elementary algebraic manipulations using the rule $\frac{d}{dt}\log(\det(a(t))=\tr(a^{-1}(t)\frac{d}{dt}a(t))$, see  \cite[Proposition II.3.3 (i)]{Faraut}.

Concerning the Fourier-Laplace transform \eqref{eq: FLT Bru}, we infer directly from their closed-form solutions
\eqref{formula: Wishartphi}--\eqref{formula: Wishartpsi}
that $\phi(t,u)$ and $\psi(t,u)$ allow for analytic extensions to the complex tube $S_d^+ +i S_d$, which we denote by $\phi(t,z)$ and $\psi(t,z)$. Hence  using analytic continuation, it can be seen that the Fourier-Laplace transform of $X$ is given by
\begin{equation}\label{eq: FLT Bru1}
\mathbb E_x[e^{-\langle z,
X_t\rangle}]=\left(\det(I+z\sigma^{\beta}_t(\alpha))\right)^{-p}
e^{-\tr\left(\left(I+z\sigma^{\beta}_t(\alpha)\right)^{-1}z
\omega^{\beta}_t(x)\right)},
\end{equation}
In order to obtain \eqref{eq: FLT Bru} from equation \eqref{eq: FLT Bru1} it suffices to observe that
for all $u,\,\theta\in M_d$ the following identity
\begin{equation}\label{lemming eq}
u(I+\theta u)^{-1}=(I+u\theta)^{-1}u
\end{equation}
holds, whenever either of both sides is well defined.
\end{proof}
Combining Lemma \ref{lemma: existence wishart dist} and Proposition \ref{prop wish proc}, we obtain the following result concerning Wishart transition kernels and -densities:
\begin{proposition}\label{prop: wishart are wishart}
Let $(X,\mathbb P_x)_{x\in S_d^+}$ be a Wishart process with parameters $(p,\alpha,\beta)$. Suppose further that the diffusion parameter $\alpha\neq 0$. Then for each $(t,x)\in \mathbb R_{++}\times S_d^+$  $X_t^x\sim\Gamma(p,\omega^{\beta}_t(x);\sigma^{\beta}_t(\alpha))$. If $\alpha\in S_d^{++}$, $p>\frac{d-1}{2}$ and $t>0$, then for all $x\in S_d^+$, $X_t^x$ has a Lebesgue density
\begin{equation*}
f_t(x,\xi)=F(p,\omega^\beta_t(x),\sigma^\beta_t(\alpha), \xi).
\end{equation*}
\end{proposition}
\subsection{On the central Wishart case}\label{subsec 3}
In the following we restate the characterization of the central
Wishart laws by using \cite{PeddadaRichards91}:
\begin{theorem}\label{th: mainproposition}
Let $d\geq 2$, $\sigma \in S_d^{++}$ and $p\geq 0$. The following
are equivalent:
\begin{enumerate}
\item \label{central1} Formula \eqref{eq: Laplacecentralwishart} is the Laplace transform of a probability measure $\Gamma(p,\omega;\sigma)$ on $S_d^+$.
\item \label{central2} $p\in \Lambda_d$.
\end{enumerate}
\end{theorem}
\begin{proof}
\ref{central1}$\Rightarrow$\ref{central2}: This is \cite[Theorem
1]{PeddadaRichards91}, as we exclude the point mass at $0$. The
converse direction is a special case of Lemma \ref{lemma: existence
wishart dist} \ref{ex wis dist 0} and \ref{ex wis dist 1}.
\end{proof}
It is important to note that condition \ref{central2} is not
necessary, if $\sigma$ is degenerate. In fact, it is easy to prove
by use of orthogonal transformations (see
http://arxiv.org/abs/1009.3708)
\begin{corollary}\label{ex: relax}
Let $r=\rank(\sigma)$. The following are equivalent:
\begin{enumerate}
\item \label{central1} Formula \eqref{eq: Laplacecentralwishart} is the Laplace transform of a probability measure $\Gamma(p,\omega;\sigma)$ on $S_d^+$.
\item \label{central2} $p\in \Lambda_r$.
\end{enumerate}
\end{corollary}
As a trivial consequence, one has \begin{corollary} The following
are equivalent:
\begin{enumerate}
\item $\Gamma(p;\sigma)$ is infinitely divisible.
\item $\rank(\sigma)=1$.
\end{enumerate}
\end{corollary}

\section{Proof of Theorem \ref{th: maintheorem}}\label{sec: proof}
Let $(p,\omega,\sigma)\in \mathbb R_{++}\times S_d^{+}\times S_d^{++}$ such that
$\mu:=\Gamma(p,\omega;\sigma)$ is a probability measure, that is,
eq.~\eqref{FLT Mayerhofer Wishart} holds. The domain of its
moment generating function is defined as
\begin{equation*}
D(\mu):=\{u\in S_d\mid \mathcal L_\mu(u):=\int_{S_d^+}e^{-\langle
u,\xi\rangle}\mu(d\xi)<\infty\},
\end{equation*}
which is the maximal domain to which the Laplace transform, originally defined for $u\in S_d^+$ only, can be extended.
It is well known that $D(\mu)$ is a convex (hence connected) set, and we also know that
$S_d^+\subset D(\mu)$. Clearly $(I+\sigma u)$ is invertible if and only if the (symmetric) matrix
$(I+\sqrt{\sigma} u \sqrt{\sigma})$ is non-degenerate. Using these facts and the defining equation
\eqref{FLT Mayerhofer Wishart} we infer that
\begin{equation}\label{equalityofsets}
D(\mu):=\{u\in S_d\mid (I+\sqrt{\sigma}
u\sqrt{\sigma})\in S_d^{++}\}=-\sigma^{-1}+S_d^{++},
\end{equation}
and therefore $D(\mu)$ is even open. Accordingly, the natural exponential family of
$\mu$ is the family of probability measures\footnote{In order to avoid confusions with calculations in the proof of the upcoming proposition, we change here from $u$ notation to $v$, because
$u$ denotes the Fourier-Laplace variable in this paper.}
\[
F(\mu)=\left.\left\{\frac{\exp(v\xi) \mu(d\xi)}{\mathcal L_\mu(v)}\right| v\in -\sigma^{-1}+S_d^{++}\right\}.
\]
We start by stating some key properties of Wishart distributions \footnote{Some related properties can be found in
Letac and Massam \cite{Letac08}, but in a different notation. More detailed information may be found in Appendix \ref{App}}:
\begin{proposition}\label{proplemma}
\begin{enumerate}
\item \label{firstelemprop} Let $p\geq 0,\,\omega\in S_d^+$. Suppose $X$ is an $S_d^+$-valued random variable distributed according to $\Gamma(p,\omega; I)$. Let $q\in S_d^+$ and set $\sigma:=q^2$.
Then $q X q \sim \Gamma(p,q\omega q;\sigma)$\footnote{Expressed in geometric language, we say that the pushfoward of $\Gamma(p,\omega; I)$ under the map
$\xi\mapsto q \xi q$ equals $\Gamma(p,q\omega q;\sigma)$.}. In particular, $\Gamma(p,\omega; I)$ exists if and only if $\Gamma(p, q \omega q;\sigma)$ exists.
\item  \label{secondelemprop} Let $p\geq 0,\, \sigma \in S_d^{++}$ and $\omega\in S_d^+$ such that $\mu:=\Gamma(p,\omega;I)$ is a probability measure.
For $v=\sigma^{-1}-I$ we have that 
\begin{equation}\label{claim 2}
\frac{\exp(v\xi) \mu(d\xi)}{\mathcal L_\mu(v)}\sim \Gamma(p, \sigma\omega\sigma;\sigma).
\end{equation}
Conversely, if  $\Gamma(p, \sigma\omega\sigma;\sigma)$ is a well defined probability measure, so is
$\mu$, and \eqref{claim 2} holds. In particular, we have that the exponential family generated by $\mu$ is a Wishart family and equals
\[
F(\mu)= \{\Gamma(p,\sigma\omega\sigma, \sigma)\mid \sigma\in S_d^{++},\quad \sigma^{-1}-I\in D(\mu)\}.
\]
\item \label{proppoint 3} Let $\Gamma(p,\omega_0;\sigma_0)$ be a probability measure, where $\sigma_0\in S_d^{++}$. Then we have
\begin{enumerate} \item \label{proppoint 3a} $\Gamma(p,t \omega_0;\sigma_0)$ is a probability measure for each $t>0$.
\item \label{proppoint 3b} If, in addition, $\omega_0$ is invertible, then $\Gamma(p, \omega;\sigma)$ is a probability measure for each $\omega\in S_d^+$, $\sigma\in S_d^+$.
\end{enumerate}
\end{enumerate}
\end{proposition}
\begin{proof}
Let $\mathbb E$ be the corresponding expectation operator. By repeated use of the cyclic property of the trace and by the product formula for the determinant, we have
\begin{align*}
\mathbb E[e^{-\langle u, q X q\rangle}]&=\mathbb E[e^{-\langle q u q, X\rangle}]=\det(I+ quq)^{-1}\exp(-\tr( quq(I+quq)^{-1}\omega))\\
&=\det(I+ \sigma u)^{-1}\exp(-\tr( uq(I+quq)^{-1}q^{-1} q \omega q))\\&=\det(I+ \sigma u)^{-1}\exp(-\tr( u (I+\sigma u)^{-1} q \omega q)),
\end{align*}
which proves assertion \ref{firstelemprop}. Next we show \ref{secondelemprop}. We note first, that by
\eqref{equalityofsets} we have that $v=\sigma^{-1}-1\in D(\mu)$. Hence exponential tilting is admissible. Furthermore, we have
\begin{equation}\label{eq: prop FLT}
\int _{S_d^+}e^{-\langle u+v,\xi\rangle}\Gamma(p,\omega; I)(d\xi)=\det(1+(u+v))^{-p}\exp(-\tr((u+v)(1+u+v)^{-1}\omega)),
\end{equation}
and setting $v=\sigma^{-1}-1$ we obtain
\[
1+u+v=\sigma^{-1}(1+\sigma u).
\]
Hence the first factor on the right side of eq.~\eqref{eq: prop FLT} is proportional to $\det(1+\sigma u)^{-p}$. It remains to show
that 
\begin{equation}\label{eq: what we would like to have}
-\tr((u+v)(1+u+v)^{-1}\omega)=c+\tr(u(1+\sigma u)^{-1} \sigma\omega\sigma)
\end{equation}
for some real constant $c$, because then the right side
of \eqref{eq: prop FLT} is proportional to the Laplace transform of $\Gamma(p,\sigma\omega\sigma;\sigma)$. To this end, we do some elementary algebraic manipulations:
\begin{align*}
-(u+v)(I+u+v)^{-1}\omega&=-(u-1+\sigma^{-1})[\sigma^{-1}(1+\sigma u)]^{-1}  \omega\\&=-(-1+\sigma^{-1}+u)(\sigma^{-1}+u)^{-1}\omega\\&=-\omega+(\sigma-\sigma)\omega+(\sigma^{-1}+u)^{-1}\omega\\&=(\sigma-I)\omega-\sigma(\sigma^{-1}+u)(\sigma^{-1}+u)^{-1}\omega+(\sigma^{-1}+u)^{-1}\omega\\&=(\sigma-I)\omega-\sigma u (\sigma^{-1}+u)^{-1}\omega\\&=(\sigma-I)\omega-\sigma u (I+\sigma u)^{-1}\sigma\omega.
\end{align*}
We set now $c:=\tr((\sigma-I)\omega)$ which is the real number we talked about before. Taking trace and performing cyclic permutation inside,
we obtain \eqref{eq: what we would like to have}, and therefore the idendity \eqref{claim 2} is shown. The assertion concerning the exponential family follows by the very definition of the latter.

We may therefore proceed to \ref{proppoint 3} which is proved by
repeatedly applying \ref{firstelemprop} and \ref{secondelemprop}: Let
$\Gamma(p,\omega_0;\sigma_0)$ be a probability measure. Then by
\ref{secondelemprop}, also $\Gamma(p,\sigma_0^{-1}\omega_0\sigma_0^{-1}; I)$ is one. Let $q_1$ such that
$q_1^2=\sigma_1\in S_d^{++}$. We may write
$\Gamma(p,\sigma_0^{-1}\omega_0\sigma_0^{-1};I)=\Gamma(p,q_1^{-1}(q_1 \sigma_0^{-1}\omega_0\sigma_0^{-1}q_1)q_1^{-1};I)$, and by
applying \ref{firstelemprop}, we obtain the pushforward measure
$\Gamma(p,q_1 \sigma_0^{-1}\omega_0\sigma_0^{-1}q_1;\sigma_1)$. By \ref{secondelemprop} we
have that $\Gamma(p, q_1^{-1}\sigma_0^{-1}\omega_0\sigma_0^{-1} q_1^{-1};I)$ is a probability measure
as well, and once again by
 \ref{secondelemprop} we infer that for all $\sigma\in S_d^{++}$, $\Gamma(p, \sigma q_1^{-1}\sigma_0^{-1}\omega_0\sigma_0^{-1} q_1^{-1}\sigma,\sigma)$
 is a probability.  We use this fact to prove both parts of the assertion. Without loss of generality we assume that
$\sigma$ is non-degenerate, because in the case $\sigma\in\partial S_d^+$ we may invoke L\'evy's continuity theorem\footnote{Strictly speaking, L\'evy's continuity theorem applies to characteristic functions. However, in the Wishart case, the right side of \eqref{FLT Mayerhofer Wishart} can even be extended to even the Fourier-Laplace transform with ease, and by preserving its functional form.}.
Setting $q_1=1/\sqrt{t} I$ and  $\sigma=\sigma_0$, we see that \ref{proppoint 3a} holds. For $\omega_0\in S_d^{++}$ we choose $q_1\in
S_d^+$ such that $q_1^{-1} \sigma_0^{-1}\omega_0\sigma_0^{-1} q_1^{-1}=\sigma^{-1}\omega\sigma^{-1}$, which allows to conclude \ref{proppoint 3b}.

\end{proof}
Finally, we deliver our proof of Theorem \ref{th: maintheorem}:
\begin{proof}
Let $p>0$ such
that for some $\omega_0\in S_d^+$, $\sigma\in S_d^{++}$, the right
side of \eqref{FLT Mayerhofer Wishart} is the Laplace transform of a
non-trivial probability measure $\Gamma(p,\omega_0;\sigma)$. By
Proposition \ref{proplemma} \ref{proppoint
3a}, we have that $\Gamma(p,\omega_0/n;\sigma)$ is a probability
measure for each $n\in\mathbb N$. Letting $n\rightarrow\infty$ and
invoking L\'evy's continuity theorem, we obtain that
$\Gamma(p;\sigma)$ is a probability measure. But then by the characterization of central Wishart laws,
Theorem \ref{th: mainproposition} \ref{central2}, we have that $p\in\Lambda_d$.

Let now $p_0\in \Lambda_d\setminus [\frac{d-1}{2},\infty)$, and let
us assume, by contradiction, that there exist $(\omega_0,\sigma)\in
S_d^+\times S_d^{++}$, $\rank(\omega_0)>2p_0+1$ such that
$\Gamma(p_0,\omega_0;\sigma)$ is a probability measure. Pick now
$\omega_1\in S_d^+$ such that $\omega^*:=\omega_1+\omega_0$ has
$\rank(\omega^*):=\rank(\omega_1)+\rank(\omega_0)=d$, and set
$p_1:=\frac{d-\rank(\omega_0)}{2}$. By construction $2p_1=
\rank(\omega_1)$, and
$p_1\in\Lambda_d\setminus[\frac{d-1}{2},\infty)$. Hence Proposition
\ref{lemma: existence wishart dist} \ref{ex wis dist 0} implies the
existence of a non-central Wishart distribution
$\Gamma(p_1,\omega_1,\sigma)$. Note that $p^*:=p_0+p_1\in
\Lambda_d\setminus[\frac{d-1}{2},\infty)$ and that by convolution
\[
\Gamma(p^*,\omega^*,\sigma):=\Gamma(p_0,\omega_0,\sigma)\star \Gamma(p_1,\omega_1,\sigma)
\]
is a probability measure as well. Since $\omega^*$ is of full rank, we have by Proposition \ref{proplemma}
\ref{proppoint 3b} that $\Gamma(p^*,\omega;\sigma)$ is a probability measure for all
$(\omega,\sigma) \in (S_d^{+})^2$. Hence $\Gamma(p^*, \omega ;t\sigma)$ exists for all $(t,\omega,\sigma)\in \mathbb R_+\times
(S_d^+)^2$.

We proceed by reverse engineering of the results of section \ref{subsec 2}. Pick any $\alpha\in
S_d^+\setminus\{0\}$. For each $(t,x)\in\mathbb R_+\times S_d^+$ we let $p_t(x,d\xi)$ be the probability measure
given by the Laplace transform \begin{equation}\label{eq: FLT Bru1} \int_{S_d^+} e^{-\langle
u,\xi\rangle}p_t(x,d\xi)=\left(\det(I+2t\alpha u) \right)^{-p^*} e^{\tr\left(-u\left(I+2t\alpha
u\right)^{-1}\,x\right)},
\end{equation}
(cf.~\eqref{eq: FLT Bru} for $\beta=0$). By the proof of Proposition
\ref{prop wish proc} we know that $\phi(t,u):=p^*\log (I+2t\alpha
u)$ and $\psi(t,u):=u\left(I+2t\alpha u\right)^{-1}$ satisfy the
system of Riccati equations \eqref{eq: phi}--\eqref{eq: psi} with
$\beta=0$. From a density argument it follows that the function
$p_t(x,d\xi)$ satisfies the Chapman-Kolmogorov equation, hence it is
the transition function of a Markov process $X$ on $S_d^+$ and by
construction the Laplace transform is exponentially affine in the
state variable $x$. Hence $X$ is an affine process in the sense of
\cite{CFMT}, with constant drift parameter $b=2p^* \alpha$ and
diffusion coefficient $\alpha$. But $b=2p^*\alpha\not\geq
(d-1)\alpha$, which contradicts the drift condition formulated in
Proposition \ref{prop: drift condition} \ref{prop: drift condition
item 1}. Therefore $\rank(\omega_0)\leq 2p_0+1$ and we have proved
the first part of the theorem.

The second part of the theorem follows from Lemma \ref{lemma:
existence wishart dist} \ref{ex wis dist 0} and \ref{ex wis dist 1}.
\end{proof}
\begin{appendix}

\section{Remarks on alternative definitions of the Wishart Distributions}\label{App}
A number of different notations and definitions of Wishart distributions appear in the literature.
This paper uses several technical tools which are presented in Letac and Massam's work \cite{Letac08}, and therefore we have chosen a notation
which is closely related to the latter. 

Letac and Massam use instead of $\Gamma(p,\omega;\sigma)$ the parameterized family $\gamma(p,a;\sigma)$, where
$\omega$ is replaced by $a:=\sigma^{-1}\omega \sigma^{-1}$. Accordingly
\eqref{FLT Mayerhofer Wishart} can be written in the form
\begin{equation}\label{FLT Letac}
\mathcal L (\gamma(p,a;\sigma))(u)= \left(\det(I+\sigma u)\right)^{-p}e^{-\tr(u(I+\sigma
u)^{-1}\sigma a\sigma)},\quad u\in S_d^+.
\end{equation} Note that this requires $\sigma$ to be invertible. Other authors
use densities to define Wishart distributions. There are, however, two notable disadvantages of using densities rather than the Laplace transform
or the characteristic function:
\begin{itemize}
\item A density need not always exist: If $\sigma$ is degenerate, $\Gamma(p, \omega;\sigma)$
 is not absolutely continuous with respect to the Lebesgue measure on $ S_d^+$. To see this
we assume for a contradiction that $\Gamma(p,\omega;\sigma)$ has a Lebesgue density, for some $\sigma$ of rank $r<d$.
Let $X$ be an $S_d^+$--valued random variable distributed according to $\Gamma(p,\omega;\sigma)$. Since linear transformations do not affect the property of having a density
and since the non-central Wishart family is invariant under linear transformations (this is easy to check), we may without loss of generality assume that $\sigma=\diag(0,I_r)$, where $I_r$ is the $r\times r$ unit matrix. Consider the projection
\[
\pi_r: x=(x_{ij})_{1\leq i,j\leq d}\mapsto \pi_r(x):=(x_{ij})_{1\leq i,j\leq r}.
\]
A simple algebraic manipulation yields that the Laplace transform of $\pi_r(X)$ equals
\[
e^{-\tr(\pi_r(\omega)v)}, \quad v\in S_{r}^+,
\]
which is the Laplace transform of the unit mass concentrated at $\pi_r(\omega)$. But the pushforward of a measure with density under a projection must have a density again. This yields the desired contradiction.
\item The density for non-central Wishart distributions is complicated, as it is a series expansion
in zonal polynomials, see \cite{Letac08}. 
\end{itemize}

As we do not work with the explicit form of the densities of non-central Wishart distributions,  we did not need to specify
the precise form of the latter in Lemma  \ref{lemma: existence wishart dist} \ref{ex wis dist 2}. 

Our use of the Laplace transform throughout the paper is due to the use of affine processes, which is a class of Markov processes with the key defining property that
their Fourier-Laplace transform is of exponentially affine form in the state variable (see, subsection \ref{subsec 2}, in particular the defining equation \eqref{eq char exp}). One of the important consequences of this
property is that the affine exponents can be determined by solving a system of ordinary differential equations, the so-called Riccati Differential Equations \eqref{eq: phi}--\eqref{eq: psi}, rather than the Forward--Kolmogorov--PDE for the Laplace transform,
\[
\partial_t \Phi(t,u,x)=\mathcal A\Phi(t,u,x),\quad \Phi(0,u,x)=\exp(-\tr(ux)),
\]
where $\mathcal A$ denotes the generator of the affine process. In other words, the last equation can be solved by an Ansatz
of the form $\Phi(t,u,x):=\exp(-\phi(t,u)-\tr(\psi(t,u))$, as it leads to the Riccati ODEs and let us avoid solving for complicated parabolic PDEs. For more technical details and a complete theory of matrix-variate affine processes we refer the interested reader to \cite{CFMT}. 

Even though specifying Wishart distributions by means of their densities is very natural, the best known construction is by
pushforwards of normal distributions under certain quadratic forms: Let $\xi_1,\xi_2,\dots,\xi_k$ be a sequence of $\mathbb R^d$--valued normally distributed
random variables with mean vectors $\mu_i\in\mathbb R^d$ and covariance matrix $\Sigma$. By using Lemma \ref{lemma:
existence wishart dist} \ref{ex wis dist minilemma} we  infer that the $S_d^+$--valued random variable $\Xi:=\sum_{i=1}^k\xi_i\xi_i^\top$
has distribution $\Gamma(p,\omega;\sigma)$, where $p=2k$, $\omega=\sum_{i=1}^k \mu_i\mu_i^\top$ and $\sigma=2\Sigma$.

In Gupta and Nagar's notation \cite{GuptaNagar01} this reads as follows. We define the $d\times k$ matrix $M=(\mu_1,\dots,\mu_k)$ and
the $d\times d$ matrix $\Theta:=\Sigma^{-1}MM^\top$. Using the matrix-variate normal distribution, we have a $d\times k$ matrix-valued random variable $X:=(\xi_1,\xi_2,\dots,\xi_k)$ which is distributed according to $\mathcal N_{d,k}(M,\Sigma\otimes I)$, and 
\begin{equation}\label{qu ex}
XX^\top\equiv \Xi\sim\mathcal W_d(k,\Sigma,\Theta),
\end{equation}
where $k$ is the shape parameter, $\Sigma$ is the scale parameter, and the parameter of non-centrality equals $\Theta$ (\cite[Theorem 3.5.1]{GuptaNagar01}). As with Letac and Massam's class of generalized non-central distributions,
this imposes that $\Sigma$ must be invertible. Accordingly, the Laplace transform of $\Xi$ is given by \cite[Theorem 3.5.3]{GuptaNagar01}
\begin{equation}\label{guptadef}
\mathbb E[e^{-\tr(u\Xi)}]=\det(I+2\Sigma u)^{-k/2}e^{-\tr(\Theta(I+2\Sigma u)^{-1}\Sigma u)}.
\end{equation}
When $\Sigma=I$ then we see that $\mathcal W_d(k,\Sigma,\Theta=MM^\top)$ equals $\Gamma(k,MM^\top;2\Sigma)$. Let us now have a look
at the general case, where $\Sigma\neq I$ and where $\Theta$ can be any positive semidefinite matrix. We observe that the right side of eq.~\eqref{guptadef} can be also written in the form 
\[
\mathbb E[e^{-\tr(u\Xi)}]=\det(I+2 u\Sigma)^{-k/2}e^{-\tr(\Theta u\Sigma (I+2u\Sigma)^{-1})},
\]
which is the notation found in \cite[equation (2)]{PeddadaRichards91}. To see this, one may use 
for the first factor the multiplicativity of the determinant, while the second factor follows from the identity
\[
[(I+2\Sigma u)^{-1}\Sigma u]^{-1}=(\Sigma u)^{-1}(I+2\Sigma u)=u^{-1}\Sigma^{-1}+2I
\]
as well as
\[
[u\Sigma (I+2 u\Sigma)^{-1}]^{-1}=\Sigma^{-1}u^{-1}+2I
\]
and the symmetry of $\Theta,\Sigma$ and $u$. 

We finally show that $\mathcal W_d(k,\Sigma,\Theta)=\Gamma(k/2,(\Theta\sigma+\sigma\Theta)/4;\sigma)$. The exponent on the right hand side of eq.~\eqref{guptadef}
can be rewritten as
\begin{align*}
\tr(\Theta(I+\sigma u)^{-1}\sigma u/2)&=\tr( (\sigma u)^{-1}+\sigma^{-1})^{-1}\Theta)=\tr( (u^{-1}+\sigma^{-1})^{-1}\Theta\sigma/2)\\&=\tr( (u^{-1}+\sigma^{-1})^{-1}\frac{\Theta\sigma+\sigma\Theta}{4})
\end{align*}
and since $\Theta,\Sigma$ are positive semidefinite, also $\Theta\sigma+\sigma\Theta$ is positive semidefinite. However, for invertible $\Sigma$, the map $X\mapsto \Sigma X+X\Sigma$
is injective \footnote{This can be seen by diagonalizing $\Sigma$.}, but not surjective \footnote{It is well known that any linear transformation
on $S_d^+$ is of the form $X\mapsto gXg^\top$, where $g$ is a real, not necessarily symmetric, $d\times d$ matrix.}, in general (unless $\Sigma$ is a multiple of the unit matrix). Hence the class of non-central Wishart distributions used in this paper is strictly larger than the classes in the standard literature, if we insist on positive semidefinite non-centrality parameters
$\Theta$. If we, however, allow (not nessesarily positive semidefinite) non-centrality parameters of the form $\Theta=2\sigma^{-1}\omega$, where $\omega\in S_d^+$, then all the three mentioned Wishart classes coincide. This definition is also in line with the quadratic construction of $\Xi$, see eq.~\eqref{qu ex}. Note, however, that \cite{PeddadaRichards91}
imposes a symmetric non-centrality parameter, which means that their class of Wishart distributions is strictly smaller than ours. Hence Corollary \ref{corollary eaton}
comprises a more general situation than \cite[Theorem 2]{PeddadaRichards91}.
\end{appendix}

\end{document}